\documentclass[12pt, a4paper, reqno]{amsart}
\usepackage{amscd, amssymb, pb-diagram,color,verbatim}
\usepackage{amsthm}
\usepackage[all]{xy}
\allowdisplaybreaks

\allowdisplaybreaks

\def\End{\operatorname{End}}
\def\Orb{\operatorname{Orb}}

\def\clsp{\operatorname{\overline{span}}}
\def\range{\operatorname{range}}

\def\ker{\operatorname{ker}}

\def\id{\operatorname{id}}

\def\id{\operatorname{id}}

\def\Tor{\operatorname{Tor}}

\def\int{\operatorname{int}}

\def\St{\operatorname{St}}

\def\C{\mathbb{C}}
\def\F{\mathbb{F}}

\def\N{\mathbb{N}}
\def\Z{\mathbb{Z}}
\def\T{\mathbb{T}}

\def\Q{\mathbb{Q}}

\def\LL{\mathcal{L}}
\def\OO{\mathcal{O}}
\def\KK{\mathcal{K}}

\def\ZZ{\mathcal{Z}}

\def\UU{\mathcal{U}}

\def\QQ{\mathcal{Q}}
\def\NO{\mathcal{N}\mathcal{O}}

\newcommand{\intfrm}[1]{\Pi{#1}}

\newcommand{\nx}{\mathbb N^{\times}}

\newcommand{\qn}{\mathcal Q_\mathbb N}

\newtheorem{thm}{Theorem}[section]

\newtheorem{cor}[thm]{Corollary}

\newtheorem{lemma}[thm]{Lemma}
\newtheorem{prop}[thm]{Proposition}

\theoremstyle{definition}

\theoremstyle{remark}
\newtheorem{remark}[thm]{Remark}

\newtheorem{example}[thm]{Example}

\numberwithin{equation}{section}

\begin{document}

\title[Exel-Larsen crossed products]{Two families of Exel-Larsen crossed products}

\author[Nathan Brownlowe]{Nathan Brownlowe}
\address{School of Mathematics and Applied Statistics\\
University of Wollongong\\
NSW 2522\\
Australia}
\email{nathanb@uow.edu.au}

\author[Iain Raeburn]{Iain Raeburn}

\address{Department of Mathematics and Statistics\\
University of Otago\\PO Box 56\\ Dunedin 9054\\
New Zealand}
\email{iraeburn@maths.otago.ac.nz}

\begin{abstract}
Larsen has recently extended Exel's construction of crossed products from single endomorphisms to  abelian semigroups of endomorphisms, and here we study two families of her crossed products. First, we look at the natural action of the multiplicative semigroup $\nx$ on a compact abelian group $\Gamma$, and the induced action on $C(\Gamma)$. We prove a uniqueness theorem for the crossed product, and we find a class of connected compact abelian groups $\Gamma$ for which the crossed product is purely infinite simple. Second, we consider some natural actions of the additive semigroup $\N^2$ on the UHF cores in $2$-graph algebras, as introduced by Yang, and confirm that these actions have properties similar to those of single endomorphisms of the core in Cuntz algebras.
\end{abstract} 

\thanks{This research has been supported by the Australian Research Council and the University of Otago. The authors thank Astrid an Huef for helpful comments and suggestions.} 
\date{\today}
\maketitle

\section{Introduction}\label{sec: intro}

In \cite{e1}, Exel considered dynamical systems $(A,\alpha,L)$ consisting of an endomorphism $\alpha$ of a unital $C^*$-algebra $A$ and a transfer operator $L$ for $\alpha$, which is a positive linear map $L:A\to A$ satisfying $L(\alpha(a)b)=aL(b)$ for $a,b\in A$. To each $(A,\alpha,L)$, he associated a crossed product $C^*$-algebra $A\rtimes_{\alpha,L}\N$. Exel showed how to realise the Cuntz-Krieger algebras as Exel crossed products \cite{e1}, and the analogous construction for nonunital $A$ includes more general graph algebras \cite{brv}.

Larsen \cite{l} has recently extended Exel's construction to actions of abelian semigroups, and we will call her algebras \emph{Exel-Larsen crossed products}. They are by definition the Cuntz-Pimsner algebras of product systems of Hilbert bimodules constructed using transfer operators. Larsen considers in particular the natural action of the multiplicative semigroup $\nx=\{a\in\Z:a>0\}$ on a compact abelian group $\Gamma$, and the associated action $\alpha:\nx\to \End C(\Gamma)$ given by $\alpha_a(f)(g)=f(g^a)$. She shows that under certain conditions (see (G1--3) in \S\ref{sec: Exel-Larsen systems} below), there is a compatible action $L$ of $\nx$ by transfer operators obtained by averaging over the solutions $h$ of $h^a=g$. When $\Gamma$ is the group $\ZZ$ of integral ad\`eles, for example, her construction yields the Hecke $C^*$-algebra of Bost and Connes \cite{bc} (see \cite[Remark~5.10]{l}); when $\Gamma$ is the circle group $\T$, $C(\T)\rtimes_{\alpha,L}\nx$ is the purely infinite simple $C^*$-algebra $\QQ_\N$ of Cuntz \cite{c2} (see \cite[Theorem~5.2]{bahlr} and \cite{hls}); and Brownlowe has shown that the $C^*$-algebras of higher-rank graphs can be obtained via a modification of her construction \cite{b}.

Here we investigate two families of Exel-Larsen crossed products. The first family are the $C(\Gamma)\rtimes_{\alpha,L}\nx$ from \cite[\S5]{l}. We find general conditions on $\Gamma$ which ensure that $C(\Gamma)\rtimes_{\alpha,L}\nx$ is purely infinite and simple, which apply in particular when $\Gamma$ is a solenoid or a higher-dimensional torus. Our main tools are Yamashita's work on the $C^*$-algebras of topological higher-rank graphs \cite{y1}, which are also Cuntz-Pimsner algebras of product systems, and the  structure theory for compact groups, including some surprisingly new results about coverings of solenoids \cite{g}. The second family of interest to us involves actions of the additive semigroup $\N^2$ on the core in the $C^*$-algebras of rank-2 graphs with a single vertex. These graphs $\Lambda_\theta$ were introduced in the context of non-selfadjoint operator algebras in \cite{kribsp}, and their $C^*$-algebras have been studied in \cite{DY, Yang1, Yang2}; the actions of $\N^2$ of interest to us were discussed in \cite{Yang1}. The papers \cite{Yang1, Yang2} suggest that, when the algebra $C^*(\Lambda_\theta)$ is simple, it behaves very much like a Cuntz algebra, which can be viewed as the graph algebra of a rank-1 graph with a single vertex. We compute the Exel crossed products for these actions, in direct parallel with a result for the Cuntz algebras in \cite[Theorem~6.6]{aHR}.

We begin with a quick review of product systems in \S\ref{sec:ps}. Then in \S\ref{sec: Exel-Larsen systems} we discuss general properties of Exel-Larsen crossed products arising from compact abelian groups. Our main general results are a uniqueness theorem and a theorem which identifies a family of compact groups whose crossed products are purely infinite and simple. In \S\ref{sec: examples}, we discuss examples, including higher-dimensional tori and solenoids. The last section contains our work on Exel-Larsen crossed products associated to $2$-graphs. We need to use a theorem of Davidson and Yang about the periodicity of $2$-graphs with a single vertex, and in an Appendix we give a short graph-theoretic proof of their theorem.

\section{Product systems of Hilbert bimodules}\label{sec:ps}
We are interested in two families of $C^*$-algebras, both of which are by definition the Cuntz-Pimsner algebras of product systems of Hilbert bimodules over a fixed $C^*$-algebra. So we begin by setting out our conventions.

A Hilbert bimodule over a $C^*$-algebra $A$ is a right Hilbert $A$-module $M$ with a left action of $A$ by adjointable operators, implemented by a homomorphism $\phi$ of $A$ into the $C^*$-algebra $\LL(M)$ of adjointable operators. (In some papers, such $M$ are described as ``right-Hilbert $A$--$A$ bimodules'', or ``correspondences over $A$.'')

Suppose that $P$ is an abelian semigroup with an identity $e$; in this paper, $P$ is either the multiplicative semigroup $\nx$ of positive integers or the additive semigroup $\N^2$. A product system over $P$ is a collection $\{M_p:p\in P\}$ of Hilbert bimodules over the same $C^*$-algebra $A$ together with an associative multiplication on $\bigsqcup_{p\in P}M_p$ such that $(m,n)\mapsto mn$ induces an isomorphism of $M_p\otimes_A M_q$ onto $M_{pq}$; we always assume that the bimodule $M_e$ is ${}_AA_A$, and that the products from $M_e\times M_p\to M_p$ and $M_p\times M_e \to M_p$ are given by the module actions of $A$ on $M_p$. 

A representation $\psi=\{\psi_p:p\in P\}$ of a product system in a $C^*$-algebra $B$ consists of linear maps $\psi_p:M_p\to B$ such that each $(\psi_p,\psi_e)$ is a representation of $M_p$, and $\psi$ is Cuntz-Pimsner covariant if each $(\psi_p,\psi_e)$ is Cuntz-Pimsner covariant. The Cuntz-Pimsner algebra\footnote{The potentially different Cuntz-Pimsner algebras $\OO(M)$ of \cite{f} and $\NO(M)$ of \cite{sy} coincide in our case, because the left actions $\phi_p:A\to \LL(M_p)$ all have range in $\KK(M_p)$ and our semigroups $\nx$ and $\N^2$ are lattice-ordered in the sense that every pair has a least upper bound (see \cite[\S5.1]{sy}).} $\OO(M)$ is generated by a universal Cuntz-Pimsner covariant representation $k_M=\{k_{M,p}\}$. These algebras were introduced by Fowler \cite{f}; we follow \cite{sy} in writing $\Pi\psi$ for the representation of $\OO(M)$ satisfying $(\Pi\psi)\circ k_{M,p}=\psi_p$.

Next we suppose that $(A,P,\alpha,L)$ is an \emph{Exel-Larsen system} as in \cite{l}: thus $\alpha:P\to \End A$ is an action of $P$ by unital endomorphisms of a unital $C^*$-algebra $A$, and $L$ is an action of $P$ by transfer operators $L_p$ for $\alpha_p$. We put a bimodule structure on each $A_{L_p}=A$ by $a\cdot x\cdot b=ax\alpha_p(b)$; give it a pre-inner product defined by $\langle x,y\rangle_{L_p}=L_p(x^*y)$; mod out by the vectors of length $0$ and complete to get a right Hilbert $A$-module $M_p$ (denoted $M_{L_p}$ in \cite{br}). The left action of $A$ on $A_{L_p}$ gives a homomorphism $\phi_p$ of $A$ into the algebra $\LL(M_p)$ of adjointable operators, and we write $q_p^M(x)$ or just $q_p(x)$ for the image of $x\in A=A_{L_p}$ in $M_p$. Larsen proves\footnote{We prefer the conventions of \cite{br} to those of \cite{l}, where Larsen completes $A\alpha_p(1)$ rather than $A$. For us, $\alpha_p(1)=1$, so we can just ignore the $\alpha_p(1)$. However, one can also ignore it in general: the completion process involves modding out by vectors of norm $0$, and every $m-m\alpha_p(1)$ has norm $0$, so $q_p(x)=q_p(x\alpha_p(1))$.} in \cite[\S3.2]{l} that $M:=\bigsqcup M_p$ is a product system of Hilbert bimodules over $A$ with 
\begin{equation}\label{mult in prod sys}
q^M_p(x)q^M_r(y)=q^M_{pr}(x\alpha_p(y)).
\end{equation}
The \emph{Exel-Larsen crossed product} $A\rtimes_{\alpha,L}P$ is by definition the Cuntz-Pimsner algebra $\OO(M)$.

In the families of Exel-Larsen systems of interest to us, we can make some simplifications. The endomorphisms $\alpha_p$ are unital, and the transfer operators $L_p$ are faithful, in the sense that $L_p(a^*a)=0\Longrightarrow a=0$. The Hilbert bimodules are finitely generated as right modules, so no modding out or completing is required. In fact our bimodules $M_p$ usually have finite orthonormal bases $\{e_i:1\leq i\leq N_p\}$, and then the identity $\phi_p(a)=\sum_{i=1}^{N_p} \Theta_{a\cdot e_i,e_i}$ implies in particular that $\phi_p$ has range in the algebra $\KK(M_p)$ of compact operators.

\section{Exel-Larsen systems from compact abelian groups}\label{sec: Exel-Larsen systems}

Let $\Gamma$ be a compact abelian group, and let $\omega$ be the action of $\nx$ by endomorphisms of $\Gamma$ given by $\omega_a(g)=g^a$ for $a\in\nx$. Let $\alpha$ denote the action of $\nx$ by endomorphisms of $C(\Gamma)$ given by $\alpha_a(f)(g)=f(\omega_a(g))=f(g^a)$. We will require that
\begin{itemize}
\item[(G1)] $\omega_a(\Gamma)$ has finite index in $\Gamma$ for all $a\in\nx$;
\item[(G2)] $\ker\omega_a$ is finite for all $a\in\nx$; and
\item[(G3)] $|\ker\omega_{ab}|=|\ker\omega_a|\cdot|\ker\omega_b|$ for all $a,b\in
\nx$.
\end{itemize}
In \cite[Proposition~5.1]{l} Larsen proved that under the conditions (G1--3),
the formula
\begin{equation}\label{the to}
L_a(f)(g):=
\begin{cases}
\frac{1}{|\ker\omega_a|}\sum_{\omega_a(h)=g}f(h) & \text{if $g\in\omega_a(\Gamma)$}\\
0 & \text{otherwise}
\end{cases}
\end{equation}
gives an action $L$ of $\nx$ by transfer operators for $\alpha$. So $(C(\Gamma),\nx,\alpha,L)$ is an Exel-Larsen system.

We can now run the system $(C(\Gamma),\nx,\alpha,L)$ through Larsen's construction. We observe as in \cite[Lemma~3.3]{LR2} that $C(\Gamma)_{L_a}$ is complete in the norm defined by the inner product. The argument of \cite[Example~2.1(b)]{aijln} shows that each $M_a:=C(\Gamma)_{L_a}$ has an orthonormal basis with $N_a:=|\Gamma/\omega_a(\Gamma)|$ elements, and hence each $\phi_a$ has range in $\KK(M_a)$. Then $C(\Gamma)\rtimes_{\alpha,L}\nx:=\OO(M)$ is generated by the universal Cuntz-Pimsner covariant representation $k_M$ of the product system $M:=\bigsqcup_{a\in \nx}M_a$.

We want to use a description of $C(\Gamma)\rtimes_{\alpha,L}\nx$ as a topological-graph algebra. The topological graph in question is a topological $\infty$-graph in the sense of \cite{y1}, where the rank-$\infty$ case is explicitly allowed. This is relevant to our situation because prime factorisation gives an isomorphism of our multiplicative semigroup $\nx$ onto the additive semigroup $\N^\infty$. We choose to suppress this isomorphism by retaining multiplicative notation, and call our graph a \emph{topological $\nx$-graph}. However, the existence of the isomorphism $\nx\cong\N^\infty$ means that the results of \cite{y1} apply (as was previously done in \cite[Remark~5.6]{y1}, for example). 

To help give a feel for our conventions, we describe a standard example. Note straightaway that our multiplicative notation means that the set of vertices in $\Lambda$ is denoted $\Lambda^1$ rather than $\Lambda^0$.  

\begin{example}
The $\nx$-graph $\Omega_{\nx}$ has vertex set $\Omega_{\nx}^1=\nx$, and the morphisms are pairs $(a,b)$ in $\nx\times\nx$ such that $a$ divides $b$. The range of $(a,b)$ is $a$, the source of $(a,b)$ is $b$, and the degree of $(a,b)$ is $a^{-1}b$.
\end{example} 

We now describe the topological $\nx$-graph associated to our system $(C(\Gamma),\nx,\alpha,L)$.

\begin{prop}\label{top inf-graph} For $a\in\nx$ let $\Gamma_a:=\{(a,g):g\in\Gamma\}$, and give $\Lambda_\Gamma:=\bigcup_{a\in\nx}\Gamma_a$ the topology in which each $\Gamma_a$ is compact and open. Define $r,s:\Lambda_\Gamma \to \Lambda^1:=\Gamma$ and $d:\Lambda_\Gamma\to\nx$ by 
\[
r(a,g)=g,\quad s(a,g)=g^a\quad\text{and}\quad d(a,g)=a.
\] 
Then $\Lambda=(\Gamma,\Lambda_\Gamma,r,s,d)$ is a topological $\nx$-graph with composition $(a,g)(b,g^a):=(ab,g)$; $\Lambda$ is row-finite and has no sources.
\end{prop}

The only nontrivial part of this assertion is that the source map $s$ is a local homeomorphism, and this follows from the next lemma.

\begin{lemma}\label{varphi_a maps}
If $\Gamma$ is a compact abelian group satisfying \textup{(G1)} and \textup{(G2)}, then each $\omega_a$ is a proper local homeomorphism.
\end{lemma}

\begin{proof}
The homomorphism $\omega_a$ factors through a continuous isomorphism of $\Gamma/\ker \omega_a$ onto $\omega_a(\Gamma)$, which is an open subgroup of $\Gamma$ because it has finite index in $\Gamma$. Since the domain of this isomorphism is compact and the range is Hausdorff, it is a homeomorphism; since the quotient map from $\Gamma$ to $\Gamma/\ker \omega_a$ is open, so is $\omega_a$. The map $\omega_a$ is proper because it is a continuous map from a compact space into a Hausdorff space.

Since $\omega_a$ is open, to see that it is a local homeomorphism it suffices to find an open neighbourhood $U$ of $e$ such that $\omega_a|_U$ is one-to-one. Since $\ker\omega_a$ is finite, we can find an open neighbourhood $U$ of $e$ such that $gU\cap hU=\emptyset$ for $h\not=g$ in $\ker\omega_a$. Suppose $h_1,h_2\in U$ and $\omega_a(h_1)=\omega_a(h_2)$. Then $h_1h_2^{-1}\in\ker\omega_a$. Since $h_1\in U$ and $h_1=h_1h_2^{-1}h_2\in h_1h_2^{-1}U$, we have $U\cap h_1h_2^{-1}U\not=\emptyset$. So $h_1h_2^{-1}=e$, and hence $h_1=h_2$. Thus $\omega_a|_U$ is one-to-one.  
\end{proof}

\begin{proof}[Proof of Proposition~\ref{top inf-graph}] 
Routine calculations show that $(\Gamma,\Lambda_\Gamma,r,s)$ is a category. The map $r$ is trivially continuous, Lemma~\ref{varphi_a maps} implies that $s$ is a local homeomorphism, and $d$ is locally constant, hence continuous. The continuity of multiplication in $\Gamma$ implies that composition is continuous, and since the product of open sets in $\Gamma$ is open, composition in $\Lambda$ is open. To see the factorisation property, suppose that $(a,g)(b,g^a)$ can be rewritten as $(a,k)(b,l)$; then taking the ranges of both sides gives $g=k$, and composability on the right forces $l=k^a=g^a$. So $\Lambda$ is a topological $\nx$-graph. The maps $r|_{\Gamma_a}$ are bijections, so $\Lambda$ is row-finite and has no sources. \end{proof}

We can now apply the construction of \cite[page~305]{y1} to the topological $\nx$-graph $\Lambda$, and thus obtain  a product system $X$ over $\nx$ of Hilbert bimodules over the algebra $C(\Lambda^1)=C(\Gamma)$. Yamashita's topological-graph algebra $\OO(\Lambda)$ is then the Cuntz-Pimsner algebra $\OO(X)$ of this product system. We want to show that our $C(\Gamma)\rtimes_{\alpha,L}\nx$ is naturally isomorphic to $\OO(\Lambda)$, and to do this we need to compare the product systems $M$ and $X$.

When we use the obvious identification of $C(\{a\}\times\Gamma)$ with $C(\Gamma)$ to view $X_a$ as a copy $\{q^X_a(x):x\in C(\Gamma)\}$ of $C(\Gamma)$, the bimodule $X_a$ is very similar to $M_a$: indeed, the underlying spaces are the same, the left and right actions are the same, and the only difference is in the inner products, where the one on $X_a$ is given by
\[
\langle q^X_a(x),q^X_a(y)\rangle(g)=
\begin{cases}
\sum_{h^a=g} \overline{x(g)}y(g) & \text{if $g\in\omega_a(\Gamma)$}\\
0 & \text{otherwise}
\end{cases}
\]
for $x,y\in C(\Gamma_a)=C(\Gamma)$; because of the way the source maps are defined in $\Lambda$, the product system structure is given by $q^X_a(x)q^X_b(y)=q^X_{ab}(x\alpha_a(y))$.

We now define $\psi_a:X_a\to M_a$ by $\psi_a(q^X_a(x))=|\ker\omega_a|^{1/2}q^M_a(x)$. Then for each fixed $a\in \nx$, $\psi_a$ is an isomorphism of Hilbert bimodules. Further, property (G3) implies that
\begin{align*}
\psi_{ab}(q^X_a(x)q^X_b(y))&=\psi_{ab}(q^X_{ab}(x\alpha_a(y)))=|\ker\omega_{ab}|^{1/2}q^M_{ab}(x\alpha_a(y))\\
&=|\ker\omega_{a}|^{1/2}|\ker\omega_{b}|^{1/2}q^M_{a}(x)q^M_b(y)=\psi_a(q^X_a(x))\psi_b(q^X_b(y)).
\end{align*}
Thus the $\{\psi_a:a\in\nx\}$ constitute an isomorphism of product systems of Hilbert bimodules, and we have:

\begin{prop}\label{isoprodsys}
Let $j_X=\{j_{X,a}\}$ and $k_M$ be the canonical Cuntz-Pimsner covariant representations of $X$ and $M$ in $\OO(\Lambda):=\OO(X)$ and $C(\Gamma)\rtimes_{\alpha,L}\nx:=\OO(M)$. Then there is an isomorphism $\Psi$ of $\OO(\Lambda)$ onto $C(\Gamma)\rtimes_{\alpha,L}\nx$ such that $\Psi\circ j_{X,a}=|\ker\omega_a|^{1/2}k_{M,a}$ for every $a\in\nx$.
\end{prop}

We now use this topological-graph realisation to prove a uniqueness theorem for the crossed products $C(\Gamma)\rtimes_{\alpha,L}\nx$. 

\begin{thm}\label{uniqueness thm}
Suppose that $\Gamma$ is a compact abelian group satisfying \textup{(G1--3)}, and that the torsion subgroup $\Tor\Gamma$ has empty interior. If $\pi:C(\Gamma)\rtimes_{\alpha, L}\nx\to B$ is a unital homomorphism into a $C^*$-algebra $B$ and $\pi\circ k_{M,1}$ is injective, then $\pi$ is injective.
\end{thm}

We will deduce Theorem~\ref{uniqueness thm} from Yamashita's uniqueness theorem for topological $k$-graphs. We recall that an infinite path in a topological $\nx$-graph $\Lambda$ with range $v$ is a degree-preserving continuous functor $\mu:\Omega_{\nx}\to\Lambda$ with $\mu(1)=v$. For $m\in \nx$, the translate $\tau^m(\mu)$ is the infinite path such that $\tau^m(\mu)(a,b):=\mu(ma,mb)$. Following Yamashita, we say that $\Lambda$ is \emph{aperiodic} if for each open set $V$ in $\Lambda^1$, there exists an infinite path $\mu\in \Lambda^\infty$ with range in $V$ such that $m\not= n$ implies $\tau^m(\mu)\not=\tau^n(\mu)$. Yamashita's uniqueness theorem is about aperiodic graphs, so the next lemma explains why we need the condition on the torsion subgroup in Theorem~\ref{uniqueness thm}.

\begin{lemma}\label{Cond A for Gamma}
Let $\Gamma$ be a compact abelian group satisfying \textup{(G1--3)}, and let $\Lambda$ be the topological $\nx$-graph in Proposition~\ref{top inf-graph}. 
\begin{enumerate}
\item\label{smallpathsp} For $g\in\Gamma$ we define $\mu_g:\Omega_{\nx}\to\Lambda$ by $\mu_g(a,b)=(a^{-1}b,g^a)$. Then $\mu_g\in\Lambda^\infty$, and  for every $\beta\in\Lambda^\infty$ we have $\beta=\mu_{r(\beta)}$. 
\item\label{equivaper} $\Lambda$ is aperiodic if and only if $\int(\Tor\Gamma)=\emptyset$.
\end{enumerate}
\end{lemma}

\begin{proof}
\eqref{smallpathsp} Fix $g\in\Gamma$. Then for composable morphisms $(a,b)$, $(b,c)$ in $\Omega_{\nx}$ we have
\begin{align*}
\mu_g(a,b)\mu_g(b,c)&=(a^{-1}b,g^a)(b^{-1}c,g^b)=(a^{-1}b,g^a)(b^{-1}c,{(g^a)}^{a^{-1}b})\\
&=(a^{-1}c,g^a)=\mu_g(a,c)=\mu_g((a,b)(b,c)),
\end{align*}
so $\mu_g$ is a functor, and $\mu$ is trivially degree-preserving. So $\mu_g$ is an infinite path in $\Lambda$. To see that they all have this form, let $\beta\in \Lambda^{\infty}$ and take $g=r(\beta)$. Then for $(a,b)\in\Omega_{\nx}$, we have $\beta(1,b)=\beta((1,a)(a,b))=\beta(1,a)\beta(a,b)$. Since $r(\beta)=g$, we have $\beta(1,b)=(b,g)$, and hence the factorisation property implies that $\beta(1,a)=(a,g)$ and $\beta(a,b)=(a^{-1}b,g^a)=\mu_g(a,b)$. Hence $\beta=\mu_g$.

\eqref{equivaper} Since each path in $\Lambda^\infty=\{\mu_g\}$ is determined by its range, we have 
\[
\tau^a(\mu_g)\not=\tau^b(\mu_g)\Longleftrightarrow r(\tau^a(\mu_g))\not=r(\tau^b(\mu_g))\Longleftrightarrow g^a\not= g^b.
\]
Thus $\Lambda$ is aperiodic if and only if each open set $V$ in $\Lambda^1=\Gamma$ contains some $g=r(\mu_g)$ with infinite order. But this is equivalent to saying that the torsion subgroup has empty interior.
\end{proof}

\begin{proof}[Proof of Theorem~\ref{uniqueness thm}]
The homomorphism $\pi$ is the integrated form $\Pi\rho$ of a Cuntz-Pimsner covariant representation $\{\rho_{M,a}:a\in \nx\}$ in $B$, and then $\rho_{M,1}=\pi\circ k_{M,1}$ is injective. Composing with the isomorphism $\Psi$ of Proposition~\ref{isoprodsys} gives a Cuntz-Pimsner covariant representation $\{\rho_{M,a}\circ \Psi\}$ of $X$ with $(\rho\circ\Psi)_1=\rho_{M,1}$ injective. Since $\Tor\Gamma$ has empty interior, Lemma~\ref{Cond A for Gamma} says that $\Lambda$ is aperiodic.  So Yamashita's uniqueness theorem \cite[Theorem~3.11]{y1} implies that $\intfrm(\rho\circ\Psi)$ is injective. Hence $\intfrm\rho=\intfrm(\rho\circ\Psi)\circ\Psi^{-1}$ is injective too.
\end{proof}

\begin{lemma}\label{minimal}
Let $\Gamma$ be a compact abelian group satisfying \textup{(G1--3)}. Then the topological graph $\Lambda$ of Proposition~\ref{top inf-graph} is minimal in the sense of \cite[4.1]{y1} if and only if 
\begin{equation}\label{mincond}
\overline{\{h\in \Gamma: h^a=g^b\text{ for some }a,b\in\nx\}}=\Gamma\text{ for every $g\in\Gamma$.}
\end{equation}   
\end{lemma}

\begin{proof}
The equivalence of (ii) and (ii) in \cite[Theorem~4.7]{y1} says that the graph $\Lambda$ is minimal if and only if 
\begin{equation}\label{orbit eq}
\Orb(g,\mu):=\bigcup_{b\in\nx}\bigcup_{a\in\nx}r|_{\Gamma_a}(s|_{\Gamma_a}^{-1}(\mu(b)))
\end{equation}
is dense in $\Lambda^1=\Gamma$ for every $g\in \Gamma$ and every $\mu\in \Lambda^\infty$ with $r(\mu)=g$. Lemma~\ref{Cond A for Gamma}\,\eqref{smallpathsp} implies that for each $g$ there is exactly one such path $\mu$, namely the path $\mu_g$. But for this path, the $b$th vertex $\mu_g(b)$ is $g^b$, and $\Orb(g,\mu_g)$ is the set whose closure appears in~\eqref{mincond}.
\end{proof}

If $G$ also satisfies $\int(\Tor\Gamma)=\emptyset$, then $\Lambda$ is aperiodic, and it follows from \cite[Theorem~4.7]{y1} that \eqref{mincond} is equivalent to simplicity of $C^*(\Lambda)=C(\Gamma)\rtimes_{\alpha,L}\nx$. For connected groups, \eqref{mincond} always holds, and we get the following.

\begin{thm}\label{simple and p.i. thm}
Suppose that $\Gamma$ is a connected compact abelian group satisfying \textup{(G2)} and $\int(\Tor\Gamma)=\emptyset$. Then $\Gamma$ also satisfies \textup{(G1)} and \textup{(G3)}, and $C(\Gamma)\rtimes_{\alpha,L}\nx$ is purely infinite and simple.   
\end{thm}

The following proof uses the following characterisations of connectedness for a compact abelian group $\Gamma$ (the second holds also for locally compact $\Gamma$):
\begin{itemize}
\item[(i)]\cite[Theorem~24.25]{hr}: $\Gamma$ is connected if and only if $\Gamma$ is divisible, in the sense that for every $g\in\Gamma$ and $a\in\nx$, there exists $h\in\Gamma$ with $h^a=g$,
\item[(ii)]\cite[Corollary~7.9]{hr}: $\Gamma$ is connected if and only if for every neighbourhood $U$ of the identity in $\Gamma$, we have $\bigcup_{n=1}^{\infty}U^n=\Gamma$.
\end{itemize}

\begin{proof}[Proof of Theorem~\ref{simple and p.i. thm}]
We again take $\Lambda$ as in Proposition~\ref{top inf-graph}. From (i) above we know that $\Gamma$ is divisible, which means each $\omega_a$ is surjective, and we have (G1). As pointed out in \cite[Section~5.1]{l}, if $\omega_a$ is surjective, then for each $b\in\nx$ the map $g\mapsto\omega_a(g)$ is a homomorphism of $\ker\omega_{ab}$ onto $\ker\omega_b$ with kernel equal to $\ker\omega_a$, and this implies (G3). 

Since $\OO(\Lambda)\cong C(\Gamma)\rtimes_{\alpha,L}\nx$, it suffices to show that $\Lambda$ satisfies the conditions of \cite[Theorem~4.13]{y1}. Since $\int(\Tor\Gamma)=\emptyset$, Lemma~\ref{Cond A for Gamma} says that $\Lambda_\Gamma$ is aperiodic. Thus we need to show that $\Lambda_\Gamma$ is minimal and contracting, in the sense of \cite[Theorem~4.7(iii)]{y1} and \cite[Definition~4.12]{y1}, respectively.

We let $g\in G$, and claim that $\{h\in \Gamma:h^a=g\text{ for some }a\in \nx\}$ is dense in $\Gamma$. Suppose $h\in\Gamma$ and $U$ is an open neighbourhood of the identity $e\in\Gamma$; we need to find $k\in hU$ and $a\in\nx$ such that $k^a=g$. We know from (ii) that $\Gamma=\bigcup_{n=1}^{\infty}U^n$. Since $\Gamma$ is compact, there exists a finite subset $F$ of $\nx$ with $\Gamma=\bigcup_{a\in F}U^a$. For $b,c\in\nx$ with $b|c$ the map $u\mapsto ue^{b^{-1}c}$ is an injection of $U^b$ in $U^c$. Thus the $\{U^b\}$ are nested, and there exists $a\in\nx$ with $\Gamma=U^a$. Now choose $h_0\in U$ such that $h_0^a=gh^{-a}$, and then $k:=hh_0$ satisfies $k\in hU$ and $k^a=h^ah_0^a=h^agh^{-a}=g$, proving the claim. Now Lemma~\ref{minimal} implies that $\Lambda$ is minimal

To see that $\Lambda $ is contracting, consider any open neighbourhood $U$ of $e$. Then as in the previous paragraph there exists $a\in\nx$ with $\Gamma=U^a$. This implies that $U$ is a contracting open set, as defined in \cite[4.11]{y1}, and hence that $\Lambda$ is contracting at $e$, as in \cite[4.12]{y1}. The result now follows from \cite[Theorem~4.13]{y1}.
\end{proof}

\section{Examples}\label{sec: examples}

We now apply the general results of the previous section to some specific compact abelian groups $\Gamma$.

\subsection{Tori} For each positive integer $l$ the torus $\T^l$ is a connected compact abelian group. Each $\omega_a$ satisfies $|\ker\omega_a|=a^l<\infty$, so (G2) holds. We now claim that $\int(\Tor\T^l)=\emptyset$, or equivalently  that the elements of infinite order are dense in $\T^l$. To see this, let $U$ be an open set in $\T^l$, and choose $w\in U$ with first coordinate of the form $e^{i\theta}$ for some rational $\theta$; then the exponent in the first coordinate of $w^a$ is never an integer multiple of $2\pi i$, and $w$ has infinite order. So Theorem~\ref{simple and p.i. thm} gives the following result.

\begin{prop}\label{the l-torus}
For each positive integer $l$, the crossed product $C(\T^l)\rtimes_{\alpha,L}\nx$ is purely infinite and simple.
\end{prop}

\begin{remark}\label{rem about l=1}
This result is known for $l=1$, because $C(\T)\rtimes_{\alpha,L}\nx$ is isomorphic to Cuntz's algebra $\QQ_\N$ (by \cite[Theorem~5.2]{bahlr}), and Cuntz proved in \cite{c2} that $\QQ_\N$ is purely infinite simple.
\end{remark}


\subsection{Solenoids}

Let $P=(p_1,p_2,\dots)$ be a sequence of primes. The \emph{$P$-adic solenoid} is the inverse limit $\Sigma_P:=\varprojlim(\T,z\mapsto z^{p_i})$, which is a compact abelian group. We will use results of Charatonik-Covarrubias \cite{cc} and Gumerov \cite{g} to study $C(\Sigma_P)\rtimes_{\alpha,L}\nx$.

The solenoid $\Sigma_P$ is  connected because it is the inverse limit of connected spaces \cite[Theorem~6.1.20]{eng}. The maps $\omega_a:\Sigma_P\to\Sigma_P$ are studied in \cite{g} (there  denoted $h_P^a$). Proposition~3 of \cite{g} says that $\omega_a$ is a homeomorphism if each prime divisor of $a$ occurs infinitely often in $P$, and Theorem~1 of \cite{g} says that $|\ker\omega_a|=a$ if and only if no prime divisor of $a$ occurs infinitely often in $P$. Let $P_1$ be the set of primes which occur infinitely often in $P$. Then each $a$ has a unique factorisation $a=bc$ where $b$ is coprime to every $p\in P_1$, and $c$ has all its prime factors in $P_1$, and we have $|\ker\omega_a|=|\ker(\omega_b\circ\omega_c)|=|\ker\omega_b|=b$. Hence  $\Sigma_P$ satisfies (G2).  

We claim that the elements of infinite order are dense in $\Sigma_P$. To see this we consider the canonical continuous homomorphism $\pi_n$ of the inverse limit $\varprojlim(\T,z\mapsto z^{p_i})$ onto the $n$th copy of $\T$. Then the sets $\{\pi_n^{-1}(U):n\geq 1,\text{ $U$ is open in $\T$}\}$ form a basis for the topology on $\Sigma_P$. If $g\in\Sigma_P$ has some $\pi_n(g)=e^{i\theta}$ with $\theta$ rational, then $g$ has infinite order, and since the $\pi_n$ are surjective, there are such points in every $\pi_n^{-1}(U)$. So the elements of infinite order are dense, as claimed, and $\int(\Tor\Sigma_P)=\emptyset$.

Thus Theorem~\ref{simple and p.i. thm} gives:

\begin{prop}\label{the solenoid}
The crossed product $C(\Sigma_P)\rtimes_{\alpha,L}\nx$ is purely infinite and simple.
\end{prop}


\subsection{The $p$-adic integers}

For each prime $p$, the $p$-adic integers are the elements of the inverse limit $\Z_p:=\varprojlim\Z/p^k\Z$, where each map $\Z/p^{k+1}\Z\to\Z/p^k\Z$ is reduction modulo $p^k$. Each $\Z/p^k\Z$ is a finite ring, and hence the inverse limit is a compact ring; we are interested in the additive group $(\Z_p,+)$, which is a compact abelian group but is not connected. For $a\in\nx$, the map $\omega_a$ takes $x$ to $ax$, and is injective. We can write $a=p^nb$, where $b$ is coprime to $p$, and then $|\Z_p:\omega_a(\Z_p)|=p^n$. So $\Z_p$ satisfies (G1--3). 

The transfer operator $L_a$ on $C(\Z_p)$ is given by 
\[
L_a(f)(x)=
\begin{cases}
f(a^{-1}x) & \text{if $x\in a\Z_p=\range\omega_a$,}\\
0 & \text{otherwise}
\end{cases}
\]
and $a\mapsto L_a$ is an action by endomorphisms of $C(\Z_p)$ such that $\alpha_a\circ L_a=\id$. So we are in the situation of \cite[Theorem~5.7]{l}, and \cite[Corollary~5.8]{l} implies that $C(\Z_p)\rtimes_{\alpha,L}\nx$ is isomorphic to the Stacey crossed product $C(\Z_p)\times_{L}^{\St}\nx$ (as in \cite{s,lr2,aHR}, for example).

We can say more about this Stacey crossed product using the analysis of \cite{lpr,blpr}. It follows from \cite[Theorem~4.3]{l2} that the ideal $C_0(\Z_p\backslash\{0\})$ in $C(\Z_p)$ is extendibly invariant in the sense of Adji \cite{a}, and hence by \cite[Theorem~1.7]{l2} gives an ideal $C_0(\Z_p\backslash\{0\})\times^{\St}_L\nx$ in $C(\Z_p)\times_{L}^{\St}\nx$ with quotient $\C\times_{L}^{\St}\nx$ isomorphic to $C^*(\nx)=C(\T^\infty)$.

To analyse the structure of the ideal, set $P:=\{b\in \nx:b\text{ is coprime to }p\}$, and note that $(n,b)\mapsto p^nb$ is an isomorphism of $\N\times P$ onto $\nx$. Similarly, with $\UU(\Z_p)$ the group of units in the ring $\Z_p$, the map $(n,x)\mapsto p^nx$ is a homeomorphism of $\N\times \UU(\Z_p)$ onto $\Z_p\backslash\{0\}$. This homeomorphism gives a tensor-product decomposition $C_0(\Z_p\backslash\{0\})=c_0(\N)\otimes C(\UU(\Z_p))$, and in this decomposition the endomorphism $\alpha_{p^na}$ decomposes as $\tau_n\otimes\sigma_a$, where 
\[
\tau_n(f)(m)=\begin{cases}
f(m-n)&\text{if $m\geq n$}\\
0&\text{if $m<n$,}
\end{cases}
\]
and $\sigma_a(g)(x)=g(ax)$ (see \cite[page~175]{lpr}). Now Theorem~2.6 of \cite{l2} implies that 
\[
\big(c_0(\N)\otimes C(\UU(\Z_p))\big)\times^{\St}_\alpha\nx\cong \big(c_0(\N)\times^{\St}_\tau\N\big)\otimes\big(C(\UU(\Z_p))\times^{\St}_\sigma P\big).
\]
As in the proof of \cite[Lemma~2.5]{lpr}, $c_0(\N)\times^{\St}_\tau\N$ is isomorphic to $\KK(\ell^2(\N))$. The endomorphisms $\sigma_a$ are in fact automorphisms, and hence $\sigma$ extends to an action $\overline{\sigma}$ of the subgroup $PP^{-1}$ of $\Q^*$. Theorem~4.1 of \cite{blpr} implies that 
\[
C(\UU(\Z_p))\times^{\St}_\sigma P=C(\UU(\Z_p))\times_{\overline{\sigma}} (PP^{-1})
\]
is a simple AT-algebra with real rank zero and a unique tracial state.


\subsection{The integral ad\`eles} The integral ad\`eles are the elements of the ring $\ZZ:=\prod_{p\textup{ prime}}\Z_p$, and the additive group $(\ZZ,+)$ satisfies (G1--3). As in the preceding example,  \cite[Corollary~5.8]{l} gives isomorphisms
\begin{equation}\label{BCisos}
C(\ZZ)\rtimes_{\alpha,L}\nx\cong C(\ZZ)\rtimes_{L,\alpha}\nx\cong C(\ZZ)\rtimes_L^{\St}\nx,
\end{equation}
and since $\ZZ$ is the Pontryagin dual of $\Q/\Z$, \cite[Corollary~2.10]{lr2} implies that the Stacey crossed product is isomorphic to the Bost-Connes algebra (as observed in \cite[Example~5.9]{l}). The statements of \cite[Lemma~4.5 and Theorem~5.7]{l} show that the isomorphisms in \eqref{BCisos} preserve the copies of $C(\ZZ)$, so Theorem~\ref{uniqueness thm} gives \cite[Theorem~3.7]{lr2}.

In view of the isomorphism \eqref{BCisos}, the results of \cite{lr3} show that the Exel-Larsen crossed product $C(\ZZ)\rtimes_{\alpha,L}\nx$ has a complicated ideal structure, and is in  particular not simple.

\begin{remark}
Of course this proof of \cite[Theorem~3.7]{lr2} goes through Yamashita's uniqueness theorem for row-finite graphs with no sources from \cite{y1}. So it is mildly curious that he was not able to do this in \cite{y1}, because the topological-graph model he was using has sources (see \cite[Remark~5.6]{y1}). He has since proved a more general uniqueness theorem which does apply to this example (see \cite{y2}, especially Example~5.17).
\end{remark}


\section{Exel-Larsen systems from higher-rank graphs}

Let $E$ be a finite directed graph with one vertex, and partition the edge set $E^1=E^1_B\sqcup E^1_R$ into blue and red edges. We suppose that there are $N_1$ blue edges and $N_2$ red edges, and we use multiindex notation: for $n=(n_1,n_2)\in\N^2$ we write $N^n:=N_1^{n_1}N_2^{n_2}$. We write $E_{BB},E_{BR},E_{RB}$ and $E_{RR}$ for the sets of blue-blue, blue-red, red-blue and red-red paths of length $2$, respectively. We suppose that we have a bijection $\theta:E_{BR}^2\to E_{RB}^2$, and we write\footnote{This is different from the notation in \cite{Yang1,Yang2}, where the sets $E^1_B$ and $E^1_R$ are listed as $e_1,\cdots,e_m$ and $f_1,\cdots, f_n$, and $\theta$ is a map on sets of labels. Yang writes ``$e_if_j=f_{j'}e_{i'}$ where $\theta(i,j)=(i',j')$''.} $\theta_1(ef)\in E^1_R$ and $\theta_2(ef)\in E^1_B$ for the  edges in $\theta(ef)$, so that $\theta(ef)=\theta_1(ef)\theta_2(ef)$.

A theorem of Kumjian and Pask \cite[\S6]{kp} says that there is a unique $2$-graph $\Lambda_\theta$ such that $\Lambda_\theta^{(1,0)}=E^1_B$, $\Lambda_\theta^{(0,1)}=E^1_R$, and the paths of degree $(1,1)$ have factorisations $ef$ and $\theta(ef)$. The $C^*$-algebra $\OO_\theta:=C^*(\Lambda_\theta)$ has been studied by Davidson and Yang \cite{DY,Yang1, Yang2}. The $C^*$-algebra $\OO_\theta$ is  generated by two Cuntz families $\{s_e:e\in E^1_B\}$ and $\{s_f:f\in E^1_R\}$ such that $s_es_f=s_{\theta_1(ef)}s_{\theta_2(ef)}$, and is universal for such families. As usual, we have 
\[
\OO_{\theta}=\clsp\{s_\lambda s_\mu^*:\lambda, \mu\in \Lambda_\theta\}.
\]
There is a gauge action $\gamma$ of $\T^2$ characterised by $\gamma_z(s_\lambda s_{\mu}^*)=z^{d(\lambda)-d(\mu)}s_\lambda s_{\mu}^*$, and the core $\OO_{\theta}^\gamma$ is the fixed-point algebra
\[
\OO_{\theta}^\gamma=\clsp\{s_\lambda s_\mu^*:d(\lambda)=d(\mu)\},
\]
which is a UHF algebra of type $(N_1N_2)^\infty$.

Yang observes in \cite[\S3.1]{Yang1} that there is an action $\alpha:\N^2\to \End\OO_\theta$ such that
\[
\alpha_n(a)=\sum_{\lambda\in \Lambda_\theta^n}s_\lambda a s_\lambda^*\quad\text{ for $a\in \OO_\theta$.}
\]
This is the analogue for the rank-2 graph algebra $\OO_\theta$ of the canonical unital endomorphism $\alpha(a)=\sum_{1=1}^n s_ias_i^*$ of the Cuntz algebra $\OO_n=C^*(s_i)$. In \cite[\S 3]{Yang1}, Yang investigates other unital endomorphisms of the $\OO_\theta$, and finds striking parallels with the known properties of endomorphisms of the Cuntz algebras. 

An Huef and Raeburn have recently considered an Exel system $(\OO_n^\gamma,\alpha,L)$ involving the restriction of the canonical endomorphism $\alpha$ to the core, and showed that $\OO_n^\gamma\rtimes_{\alpha,L}\N$ is the Cuntz algebra $\OO_{n^2}$. (To recover this from \cite[Theorem~6.6]{aHR}, take $N=n$ in \cite[\S6]{aHR}, so that the projection $p$ appearing there is the identity of $M_n(\C)$.) Here we show that the analogous Exel-Larsen system for Yang's endomorphic action has similar properties: the crossed product is another $\OO_\eta$ associated to a larger graph $\Lambda_\eta$. We first need an action by transfer operators.

\begin{prop}
For each $n\in \N^2$, the function $L_n:\OO_\theta\to \OO_\theta$ defined by 
\[
L_n(a)=\frac{1}{N^n}\sum_{\lambda\in \Lambda_\theta^n}s_\lambda^* a s_\lambda
\]
is a transfer operator for $\alpha_n$ such that $L_n(1)=1$, and we then have $L_m\circ L_n=L_{m+n}$.
\end{prop}

\begin{proof}
The map $L$ is clearly positive and linear, and $L_n(1)=1$ because each $s_\lambda$ is an isometry and there are $N^n$ summands. For $a,b\in \OO_\theta$ we have
\[
L_n(\alpha_n(a)b)=\frac{1}{N^n}\sum_{\lambda\in \Lambda_\theta^n}s_\lambda^*\Big(\sum_{\mu\in \Lambda_\theta^n}s_\mu a s_\mu^*\Big)bs_\lambda.
\]
Since $\lambda$ and $\mu$ have the same degree, $s_\lambda^*s_\mu$ vanishes unless $\mu=\lambda$. Thus the inside sum has just one nonzero term,
\[
L_n(\alpha_n(a)b)=\frac{1}{N^n}\sum_{\lambda\in \Lambda_\theta^n}a s_\lambda^*bs_\lambda,
\]
 and factoring out $a$ gives $L_n(\alpha_n(a)b)=aL_n(b)$. So $L_n$ is a transfer operator for $\alpha_n$.
 
To see that $L_m\circ L_n=L_{m+n}$, we take $a\in \OO_\theta$ and calculate:
 \begin{equation}\label{Laction}
 L_m\circ L_n(a)= \frac{1}{N^m}\sum_{\mu\in\Lambda_{\theta}^m}s_\mu^*\Big(\frac{1}{N^n}\sum_{\lambda\in\Lambda_{\theta}^n}s_\lambda^* as_\lambda\Big)s_\mu=\frac{1}{N^{m+n}}\sum_{\substack{\mu\in\Lambda_{\theta}^m \\ \lambda\in\Lambda_{\theta}^n}}s_{\mu\lambda}^*as_{\mu\lambda}. 
 \end{equation}
The factorisation property implies that $\{\mu\lambda:\mu\in \Lambda_\theta^m,\ \lambda\in\Lambda_{\theta}^n\}$ is just $\Lambda^{m+n}_\theta$, and hence \eqref{Laction} implies that  $L_m\circ L_n(a)=L_{m+n}(a)$. 
\end{proof}

Both $\alpha_n$ and $L_n$ map the core $\OO_\theta^\gamma$ into itself, and hence $(\OO_\theta^\gamma,\N^2,\alpha,L)$ is an Exel-Larsen system. We now describe the corresponding crossed product $\OO_\theta^\gamma\rtimes_{\alpha,L}\N^2$. Recall that we are writing $k_M=\{k_{M,n}:n\in \N^2\}$ for the canonical Cuntz-Pimsner covariant representation of $M$ in $\OO(M)=\OO_\theta^\gamma\rtimes_{\alpha,L}\N^2$.   

\begin{thm}\label{ExelcpofOtheta}
Consider the coloured directed graph $F=(F^0,F^1=F^1_B\sqcup F^1_R)$ with $F^0=E^0$, $F^1_B=E_{BB}^2$ and $F^1_R=E_{RR}^2$, and define $\eta:F^2_{BR}\to F^2_{RB}$ by
\begin{equation}\label{defeta}
\eta((ef)(gh))=(\theta_1(eg)\theta_1(fh))(\theta_2(eg)\theta_2(fh)).
\end{equation}
Then $\eta$ is a bijection. If $\{t_{ef}:ef\in F^1_B\}$ and $\{t_{gh}:gh\in F^1_R\}$ denote the Cuntz families which generate $\OO_\eta$, then there is an isomorphism $\pi$ of $\OO_{\eta}$ onto $\OO_\theta^\gamma\rtimes_{\alpha,L}\N^2$ such that
\[
\pi(t_{ef})=k_{M,(1,0)}(N_1^{1/2}q_{(1,0)}(s_es_f^*))\quad\text{and}\quad \pi(t_{gh})=k_{M,(0,1)}(N_2^{1/2}q_{(0,1)}(s_gs_h^*)),
\]
for $ef\in F^1_B$ and $gh\in F^1_R$.
\end{thm}

We first show that $\eta$ is a bijection.  If $\eta((ef)(gh))=\eta((e'f')(g'h'))$, then repeated applications of the factorisation property in $\Lambda_\theta$ show that $e=e'$, $f=f'$, $g=g'$ and $h=h'$, whence $ef=e'f'$ and $gh=g'h'$. So $\eta$ is an injection of $F^2_{BR}=E^2_{BB}E^2_{RR}$ into $F^2_{RB}=E^2_{RR}E^2_{BB}$. Since both sets have $N_1^2N_2^2$ elements, $\eta$ is a bijection. 

In view of \cite[Lemma~2.6]{EaHR} and \cite[Lemma~6.3]{aHR}, it is not surprising that the right Hilbert $\OO_\theta^\gamma$-modules $M_n$ have orthonormal bases.

\begin{prop}
For every $n\in \N^2$, $\big\{m_{\mu\nu}^n:={N^{n/2}}q_n(s_\mu s_\nu^*):\mu,\nu\in \Lambda_\theta^n\big\}$ is an orthonormal basis for the right Hilbert $\OO_\theta^\gamma$-module $M_n$. The multiplication in $M$ satisfies 
\begin{equation}\label{prodbasis}
m_{\mu\nu}^m m_{\alpha\beta}^n=m_{(\mu\alpha)(\nu\beta)}^{m+n}.
\end{equation}
\end{prop}

\begin{proof}
Fix $\mu,\nu\in \Lambda_\theta^m$ and $\alpha,\beta\in \Lambda_\theta^n$, and recall from \eqref{mult in prod sys} that the multiplication in $M$ satisfies
\[
q_m(s_{\mu}s_{\nu}^*)q_n(s_\alpha s_\beta^*)=q_{m+n}(s_{\mu}s_{\nu}^*\alpha_m(s_\alpha s_\beta^*)).
\]
We have
\[
s_{\mu}s_{\nu}^*\alpha_m(s_\alpha s_\beta^*)=s_{\mu}s_{\nu}^*\Big(\sum_{d(\lambda)=m}s_\lambda s_\alpha s_\beta^*s_\lambda^*\Big)=\sum_{d(\lambda)=m}s_{\mu}s_{\nu}^*s_\lambda s_\alpha s_\beta^*s_\lambda^*;
\]
since $d(\nu)=m=d(\lambda)$ and the $\{s_\lambda:d(\lambda)=m\}$ are isometries, only the term with $\lambda=\nu$ survives, and we have
\begin{equation}\label{compprod}
s_{\mu}s_{\nu}^*\alpha_m(s_\alpha s_\beta^*)=s_{\mu}s_\alpha s_\beta^*s_\nu^*=s_{\mu\alpha}s_{\nu\beta}^*.
\end{equation}
This immediately gives the formula \eqref{prodbasis}, and since
\[
q_m(s_{\mu}s_{\nu}^*)\cdot s_\alpha s_\beta^*=q_m(s_{\mu}s_{\nu}^*\alpha_m(s_\alpha s_\beta^*)),
\]
\eqref{compprod} also implies that $\{m^m_{\mu\nu}:\mu,\nu\in\Lambda_\theta^m\}$ generates $M_m$ as a right Hilbert $\OO_\theta^\gamma$-module.

To check that $\{m_{\mu\nu}^n\}$ is orthonormal, we take $\mu,\nu,\alpha,\beta\in \Lambda_\theta^n$, and compute
\[
\langle m_{\mu\nu}^n, m_{\alpha\beta}^n\rangle
=N^nL_n((s_\mu s_\nu^*)^*(s_\alpha s_\beta^*))
=\sum_{d(\lambda)=n}s_{\lambda}^*s_\nu s_\mu^*s_\alpha s_\beta^*s_\lambda.
\]
Since all the paths have the same degree $n$, all the terms vanish unless $\mu=\alpha$. When $\mu=\alpha$, only the term with $\lambda=\nu$ survives, and it is zero unless $\beta=\nu$. So the inner product is zero unless $(\mu,\nu)=(\alpha,\beta)$, and then is $1$ because all the $s_\mu$ are isometries.
\end{proof}

As in \cite{EaHR} and \cite{aHR}, the orthonormal bases for $M_n$ give rise to  Cuntz families in $\OO(M)$:

\begin{prop}\label{bigCKfamily}
For $ef\in F^1_B=E^2_{BB}$ and $gh\in F^1_R=E^2_{RR}$, we set
\[
T_{ef}:=k_{M,(1,0)}(m_{ef}^{(1,0)})\quad\text{and}\quad T_{gh}:=k_{M,(0,1)}(m_{gh}^{(0,1)}).
\]
Then $\{T_{ef}:ef\in F^1_B\}$ and $\{T_{gh}:gh\in F^1_R\}$ are Cuntz families, and they satisfy
\begin{equation}\label{comrel}
T_{ef}T_{gh}=T_{\eta_1((ef)(gh))}T_{\eta_2((ef)(gh))}.
\end{equation}
\end{prop}

\begin{proof}
Since 
\[
T_{ef}^*T_{ef}=k_{M,(1,0)}(m_{ef}^{(1,0)})^*k_{M,(1,0)}(m_{ef}^{(1,0)})=k_{M,0}(\langle m_{ef}^{(1,0)},m_{ef}^{(1,0)}\rangle)=k_{M,0}(1),
\]
and since $k_{M,0}$ is unital \cite[Corollary~3.5]{brv}, each $T_{ef}$ is an isometry. The reconstruction formula for the orthonormal basis $\{m^{(1,0)}_{ef}\}$ implies that 
\[
\phi_{(1,0)}(1)=1=\sum_{ef\in F^1_B}\Theta_{m^{(1,0)}_{ef},m^{(1,0)}_{ef}},
\]
and hence the Cuntz-Pimsner covariance of $(k_{M,(1,0)},k_{M,0})$ implies that
\[
k_{M,0}(1)=k_{M,0}(\phi_{(1,0)}(1))=\sum_{ef\in F^1_B}k_{M,(1,0)}(m_{ef}^{(1,0)})k_{M,(1,0)}(m_{ef}^{(1,0)})^*=\sum_{ef\in F^1_B}T_{ef}T_{ef}^*.
\]
Thus $\{T_{ef}\}$ is a Cuntz family. The same arguments work for $\{T_{gh}\}$.

To establish \eqref{comrel}, we calculate using \eqref{prodbasis}:
\begin{align*}
T_{ef}T_{gh}&=k_{M,(1,0)}\big(m_{ef}^{(1,0)}\big)k_{M,(0,1)}\big(m_{gh}^{(0,1)}\big)\\
&=k_{M,(1,1)}\big(m_{(eg)(fh)}^{(1,1)}\big)\\
&=N^{(1,1)}k_{M,(1,1)}\big(q_{(1,1)}(s_es_gs_h^*s_f^*)\big)\\
&=N^{(1,1)}k_{M,(1,1)}\big(q_{(1,1)}(s_{\theta_1(eg)}s_{\theta_2(eg)}s_{\theta_2(fh)}^*s_{\theta_1(fh)}^*)\big)\\
&=N^{(1,1)}k_{M,(1,1)}\big(q_{(1,1)}(s_{\theta_1(eg)\theta_2(eg)}s_{\theta_1(fh)\theta_2(fh)}^*)\big)\\
&=k_{M,(1,1)}\big(m_{\theta_1(eg)\theta_2(eg),\theta_1(fh)\theta_2(fh)}^{(1,1)}\big)\\
&=k_{M,(0,1)}\big(m_{\theta_1(eg)\theta_1(fh)}^{(0,1)}\big)k_{M,(1,0)}\big(m_{\theta_2(eg)\theta_2(fh)}^{(1,0)}\big)\\
&=T_{\eta_1((ef)(gh))}T_{\eta_2((ef)(gh))}.\qedhere
\end{align*}
\end{proof}

Proposition~\ref{bigCKfamily} and the universal property of $\OO_\eta$ imply that there is a homomorphism $\pi:=\pi_T:\OO_{\eta}\to \OO_\theta^\gamma\rtimes_{\alpha,L}\N^2$ such that
\begin{align*}
\pi(t_{ef})&=T_{ef}=k_{M,(1,0)}({N_1^{1/2}}q_{(1,0)}(s_es_f^*))\quad\text{for $ef\in F^1_B$, and }\\
\pi(t_{gh})&=T_{gh}=k_{M,(0,1)}({N_2^{1/2}}q_{(0,1)}(s_gs_h^*))\quad\text{for $gh\in F^1_R$.}
\end{align*}
We complete the proof of Theorem~\ref{ExelcpofOtheta} by proving the following Proposition.

\begin{prop}
The homomorphism $\pi$ is an isomorphism of $\OO_{\eta}$ onto $\OO_\theta^\gamma\rtimes_{\alpha,L}\N^2$.
\end{prop}

\begin{proof}
The homomorphism $\pi$ is equivariant for the gauge action of $\T^2$ on $\OO_{\eta}=C^*(\Lambda_\eta)$ and the gauge action of $\T^2$ on $\OO_\theta^\gamma\rtimes_{\alpha,L}\N^2=\OO(M)$. Because each $T_{ef}$ is an isometry, the homomorphism $\pi$ is unital, and hence the image of the only vertex projection $p_v=1_{C^*(\Lambda_{\eta})}$ is nonzero. Thus the gauge-invariant uniqueness theorem \cite[Theorem~3.4]{kp} implies that $\pi$ is injective.

The formula \eqref{prodbasis} implies that the image of $\pi$ contains the image $k_{M,n}(m^n_{\mu\nu})$ of every basis element $m^n_{\mu\nu}$ for $M_n$. Since these elements generate $k_{M,n}(M_n)$ as a $k_{M,0}(\OO_\theta^\gamma)$-module, we have $k_{M,n}(M_n)\subset\pi(\OO_\eta)$ for each $n\in\N^2$. So to see that $\pi$ is surjective it suffices to show that the range of $\pi$ contains the image $k_{M,0}(\OO_\theta^\gamma)$ of the coefficient algebra; since the range is a $C^*$-subalgebra, it suffices to show that the range contains each $k_{M,0}(s_\mu s_\nu^*)$ with $d(\mu)=d(\nu)$. So fix such an element with $d(\mu)=d(\nu)=n$, say. We claim that $\phi_{n}(s_\mu s_\nu^*)\in \LL(M_n)$ satisfies
\begin{equation}\label{idleftaction}
\phi_{n}(s_\mu s_\nu^*)=\sum_{d(\lambda)=n}\Theta_{m^n_{\mu\lambda},m^n_{\nu\lambda}}.
\end{equation}

Both sides of \eqref{idleftaction} are adjointable operators on the right Hilbert $\OO_\theta^\gamma$-module $M_n$, and in particular are right $\OO_\theta^\gamma$-module homomorphisms, so it suffices to check that they agree on a typical basis element $m^n_{\alpha\beta}$. We have
\begin{align*}
\phi_n(s_\mu s_\nu^*)(m^n_{\alpha\beta})&={N^{n/2}}s_\mu s_\nu^*q_n(s_\alpha s_\beta^*)={N^{n/2}}q_n(s_\mu s_\nu^*s_\alpha s_\beta^*)\\
&=\begin{cases}
0&\text{if $\nu\not=\alpha$}\\
{N^{n/2}}q_n(s_\mu s_\beta^*)=m^n_{\mu\beta}&\text{if $\nu=\alpha.$}
\end{cases}
\end{align*}
On the other hand, we have 
\[
\sum_{d(\lambda)=n}\Theta_{m^n_{\mu\lambda},m^n_{\nu\lambda}}(m^n_{\alpha\beta})=\sum_{d(\lambda)=n}m^n_{\mu\lambda}\cdot\langle m^n_{\nu\lambda},m^n_{\alpha\beta}\rangle.
\]
The inner product vanishes unless $(\nu,\lambda)=(\alpha,\beta)$, and then is the identity in $\OO_\theta^\gamma$. So every summand vanishes unless $\nu=\alpha$, and then only the $\lambda=\beta$ summand survives, so
\[
\sum_{d(\lambda)=n}\Theta_{m^n_{\mu\lambda},m^n_{\nu\lambda}}\big(m^n_{\alpha\beta}\big)=m_{\mu\beta}^n\cdot 1=m_{\mu\beta}^n.
\]
Thus we have \eqref{idleftaction}, as claimed.

In view of \eqref{idleftaction}, the Cuntz-Pimsner covariance of $k_M$ implies that
\[
k_{M,0}(s_\mu s_\nu^*)=(k_{M,n},k_{M,0})^{(1)}(\phi^M_{n}(s_\mu s_\nu^*))=\sum_{d(\lambda)=n}k_{M,n}\big(m^n_{\mu\lambda}\big)k_{M,n}\big(m^n_{\nu\lambda}\big)^*.
\]
Since we have already seen that each $k_{M,n}\big(m^n_{\mu\lambda}\big)$ belongs to the range of $\pi$, we deduce that $k_{M,0}(s_\mu s_\nu^*)$ belongs to the range of $\pi$, and $\pi$ is surjective.
\end{proof}

Davidson and Yang have found a necessary and sufficient condition for aperiodicity of $\OO(\Lambda_\theta)$ \cite[Theorem~3.1]{DY} (and we give a new proof of their result in an appendix). It is interesting to note that their condition is not easy to verify, and they discuss at some length an algorithm for determining whether a given $\Lambda_\theta$ is periodic.  However, their result gives the following criteria for simplicity of our crossed products.

\begin{cor}\label{simple and pi for core system}
Let $\Lambda_\theta$ be a $2$-graph with one vertex such that the corresponding directed graph has $N_1\ge 2$ blue edges and $N_2\ge 2$ red edges, and let $\eta$ be the bijection described in \eqref{defeta}. The Exel-Larsen crossed product $\OO_\theta^\gamma\rtimes_{\alpha,L}\N^2$ is simple if and only if, for every pair of integers $a$ and $b$ with $N_1^a=N_2^b$ and every bijection $\rho:\Lambda_{\eta}^{(a,0)}\to\Lambda_{\eta}^{(0,b)}$, there exists $\mu\in\Lambda_{\eta}^{(a,0)}$ and $\lambda\in\Lambda_{\eta}^{(0,b)}$ with $\mu\lambda\not=\rho(\mu)\rho^{-1}(\lambda)$. If $\OO_\theta^\gamma\rtimes_{\alpha,L}\N^2$ is simple, then it is purely infinite.
\end{cor}

\begin{proof}
First note that $N_1^a=N_2^b$ if and only if $(N_1^2)^{a}=(N_2^2)^b$, so we can apply \cite[Theorem~3.4]{DY} (or Theorem~\ref{DYthm} below) to the graph $\OO_\eta$, and deduce that $\OO_\eta$ is aperiodic if and only if the condition on the bijections $\Lambda_{\eta}^{(a,0)}\to\Lambda_{\eta}^{(0,b)}$ holds.  Since graphs with a single vertex are trivially cofinal, \cite[Theorem~3.1 and Lemma~3.2]{rs} imply that $\OO_\eta$ is simple if and only if the condition in the theorem holds. Since the hypothesis of \cite[Proposition~8.8]{sims} trivially holds in our graphs, we know from that result that $\OO_\eta$ is purely infinite. Thus the result follows from Theorem~\ref{ExelcpofOtheta}.
\end{proof}

\begin{remark}
If $\ln N_1$ and $\ln N_2$ are not rationally related, then there are no such bijections $\rho$, our condition is trivially satisfied, and Corollary~\ref{simple and pi for core system} says that $\OO_\theta^\gamma\rtimes_{\alpha,L}\N^2$ is purely infinite and simple.
\end{remark}

\appendix

\section{On Davidson and Yang's characterisation of periodicity}

In \cite[Theorem~3.1]{DY}, Davidson and Yang find a necessary and sufficient condition for a $2$-graph with one vertex to be periodic. In their proof, they view the graph as a semigroup $\F_\theta^+$, and use some concrete isometric representations of $\F_\theta^+$ which they and Power had previously constructed in \cite{DPY}. Here we give a short proof of their theorem that argues directly in terms of the graph. We use the finite-path formulation of periodicity from \cite{rs}.

\begin{thm}[Davidson and Yang]\label{DYthm}
Let $\Lambda$ be a finite $2$-graph with a single vertex and at least two edges of each colour. Then $\Lambda$ is periodic if and only if there are positive integers $a,b$ and a bijection $\gamma$ of $\Lambda^{(a,0)}$ onto $\Lambda^{(0,b)}$ such that, for each $(\mu,\nu) \in \Lambda^{(a,0)}\times\Lambda^{(0,b)}$, the path $\mu\nu$ in $\Lambda^{(a,b)}$ has $(0,b)+(a,0)$ factorisation $\gamma(\mu)\gamma^{-1}(\nu)$.
\end{thm}

\begin{proof}
Suppose that $\Lambda$ is periodic. Then there exist $m,n\in \N^2$ such that $m\not=n$ and
\begin{equation}\label{defperiodic}
\mu(m,m+d(\mu)-(m\vee n))=\mu(n,n+d(\mu)-(m\vee n))\quad\text{for all $\mu$ with $d(\mu)\geq m\vee n$.}
\end{equation}
We claim that we cannot have $m\leq n$. To see this, suppose that $m\leq n$ and $\lambda$ is any path with $d(\lambda)=n$. Then our hypotheses imply that there are at least $2$ paths of degree $n-m$, and in particular there is one $\nu$ which is not equal to $\lambda(m,n)$. Then $\mu:=\lambda\nu$ does not satisfy \eqref{defperiodic}. So we cannot have one of $m$ or $n$ larger than the other, and (possibly after swapping $m$ and $n$) there exist $a$ and $b$ positive such that $m=(m\wedge n )+(a,0)$ and $n=(m\wedge n )+(0,b)$. Now taking a path $\nu$ with $d(\nu)=m\wedge n$ and applying \eqref{defperiodic} to $\mu=\nu\lambda$ shows that
\begin{equation}\label{easierperiodic}
\lambda((a,0),d(\lambda)-(0,b))=\lambda((0,b),d(\lambda)-(a,0))\quad\text{for all $\lambda$ with $d(\lambda)\geq (a,b)$.}
\end{equation}

We next claim that for each $\mu\in \Lambda^{(a,0)}$, there is a unique $\gamma(\mu)\in \Lambda^{(0,b)}$ such that
\begin{equation}\label{defgamma}
(\alpha\mu)((a,0),(a,b))=\gamma(\mu)=(\mu\beta)((0,0),(0,b))\quad\text{for all $\alpha,\beta\in\Lambda^{(0,b)}$.}
\end{equation}
The uniqueness in the factorisation property implies that there is at most one such $\gamma(\mu)$. To see there is one, take $\alpha,\beta\in\Lambda^{(0,b)}$, and consider $\lambda:=\alpha\mu\beta\in\Lambda^{(a,2b)}$. Then
\begin{align*}
(\alpha\mu)((a,0),(a,b))&=(\alpha\mu\beta)((a,0),(a,b))\\
&=(\alpha\mu\beta)((0,b),(0,2b))\qquad\text{by \eqref{easierperiodic}}\\
&=(\mu\beta)((0,0),(0,b)).
\end{align*}
Since the first term is independent of $\beta$ and the last term is independent of $\alpha$, the claim follows.

Swapping the roles of blue and red paths in the argument of the preceding paragraph shows that for each $\nu\in \Lambda^{(0,b)}$, there is a unique $\delta(\nu)\in \Lambda^{(a,0)}$ such that
\begin{equation}\label{defdelta}
(\tau\nu)((0,b),(a,b))=\delta(\nu)=(\nu\sigma)((0,0),(a,0))\quad\text{for all $\tau,\sigma\in\Lambda^{(a,0)}$.}
\end{equation} 
We claim that $\delta$ is an inverse for $\gamma$. Take $\mu\in \Lambda^{(a,0)}$. To compute $\gamma(\mu)$, we pick any $\alpha\in\Lambda^{(0,b)}$, and then \eqref{defgamma} gives $\gamma(\mu)=(\alpha\mu)((a,0),(a,b))$. To compute $\delta(\gamma(\mu))$, we take $\tau=(\alpha\mu)((0,0),(a,0))$. Then $\tau\gamma(\mu)=\alpha\mu$, and \eqref{defdelta} gives
\[
\delta(\gamma(\mu))=(\tau\gamma(\mu))((0,b),(a,b))=(\alpha\mu)((0,b),(a,b))=\mu.
\]
A similar argument shows that $\gamma\circ\delta$ is the identity. We have now proved that $\gamma$ is a bijection with $\gamma^{-1}=\delta$.

We next need to show that $\mu\nu=\gamma(\mu)\delta(\nu)$. Take $\alpha\in\Lambda^{(0,b)}$ and $\sigma\in\Lambda^{(a,0)}$, and consider the path $\lambda:=\alpha\mu\nu\sigma$: the diagram in Figure~\ref{figpath}\begin{figure}[h]
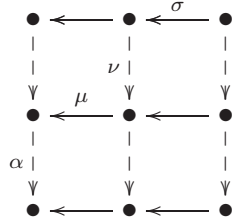

\centerline{
\xygraph{
{\bullet}="v43":[l]{\bullet}="v33":[l]{\bullet}="v23"
"v43":@{-->}[d]{\bullet}="v42":@{-->}[d]{\bullet}="v41"
"v33":@{-->}[d]{\bullet}="v32"_\nu:@{-->}[d]{\bullet}="v31"
"v23":@{-->}[d]{\bullet}="v22":@{-->}[d]{\bullet}="v21"_\alpha
"v43":"v33"_\sigma:"v23"
"v42":"v32":"v22"_\mu
"v41":"v31":"v21"
}}
\caption{The path $\lambda$  of degree $(2a,2b)$.}\label{figpath}
\end{figure}
should help see what is happening: the unbroken arrows represent blue paths of length $b$ and the dashed arrows red paths of length $a$. Periodicity in the form of \eqref{easierperiodic} tells us that the top left-hand corner in Figure~\ref{figpath}, which is $\mu\nu=\lambda((0,b),(a,2b))$, is the same as the bottom right-hand corner
\begin{align*}
\lambda((a,0),(2a,b))&=(\alpha\mu\nu\sigma)((a,0),(a,b))(\alpha\mu\nu\sigma)((a,b)(2a,b))\\
&=(\alpha\mu)((a,0),(a,b))(\nu\sigma)((0,0),(a,0)),
\end{align*}
which by \eqref{defgamma} and \eqref{defdelta} is precisely $\gamma(\mu)\delta(\nu)$. Thus $\gamma$ has all the required properties.

Now suppose that we have $a$, $b$ and $\gamma$ as in the theorem. Because $\gamma$ is a bijection, we also have 
\begin{equation}\label{altfactor}
\alpha\beta=\gamma^{-1}(\alpha)\gamma(\beta)\quad\text{for $\alpha\in\Lambda^{(0,b)}$, $\beta\in\Lambda^{(a,0)}$.}
\end{equation}
We will prove that $\Lambda$ is periodic by showing that every $\lambda\in\Lambda$ with $d(\lambda)\geq (a,b)$ satisfies \eqref{easierperiodic}. Take such a $\lambda$. Because it suffices to prove \eqref{easierperiodic} for a longer path $\lambda\mu$, we may suppose that $d(\lambda)=(Ma,Nb)$ for some positive integers $M,N$, and then it suffices to prove that
\begin{equation}\label{keybits}
\lambda\big((ma,(n+1)b),((m+1)a,(n+2)b)\big)=\lambda\big(((m+1)a,nb),((m+2)a,(n+1)b)\big)
\end{equation}
for every pair $m,n$ satisfying $0\leq m\leq M-2$ and $0\leq n\leq N-2$. Factor
\[
\lambda\big((ma,nb),((m+2)a,(n+2)b)\big)=\alpha\mu\nu\sigma\quad \text{with $\mu,\sigma\in\Lambda^{(a,0)}$ and $\alpha,\nu\in\Lambda^{(0,b)}$,}
\]
and then we are back in the situation of Figure~\ref{figpath}. The top left corner of Figure~\ref{figpath} has factorisations $\mu\nu=\gamma(\mu)\gamma^{-1}(\nu)$. The relation \eqref{altfactor} implies that the top right and bottom left corners have factorisations $\nu\sigma=\gamma^{-1}(\nu)\gamma(\sigma)$ and $\alpha\mu=\gamma^{-1}(\alpha)\gamma(\mu)$. Together, these imply that the bottom right-hand corner is $\gamma(\mu)\gamma^{-1}(\nu)$, which is the same as the top left corner. This says that $\lambda$ satisfies \eqref{keybits}, as required.
\end{proof}

\end{document}